%% file: ConjugationMain.tex
\documentclass[a4paper, american, 11pt, reqno]{amsart}

\usepackage[latin1]{inputenc}
\usepackage{mathabx}
\usepackage{amsmath}
\usepackage{amssymb}
\usepackage{hyperref}
\usepackage{url}
\usepackage{babel}
\usepackage{pgf,tikz,pgfplots}

\usepackage{tikz-cd}

\usetikzlibrary{decorations.pathmorphing}

\usepackage{graphicx}

\usepackage{tikz, float} \usetikzlibrary {positioning}
\usetikzlibrary{shapes.misc}
\usetikzlibrary{arrows}

\usetikzlibrary{calc}
\usepackage{amsfonts}
\input xy
\xyoption{all}
\newcommand{\Z}{{\mathbb Z}}

\newcommand{\C}{{\mathbb C}}

\newcommand{\R}{{\mathbb R}}

\newcommand{\CC}{{\mathcal C}}

\DeclareMathOperator{\rank}{rank}

\DeclareMathOperator{\id}{id}

\newtheorem{thm}{Theorem}[section]

\newtheorem{proposition}[thm]{Proposition}
\newtheorem{lemma}[thm]{Lemma}

\newtheorem{remark}[thm]{Remark}

\newtheorem{conj}[thm]{Conjecture}
          {\theoremstyle{definition}
}
          {\theoremstyle{definition}

}

 \numberwithin{equation}{section}

\tikzset{%
  add/.style args={#1 and #2}{to path={%
 ($(\tikztostart)!-#1!(\tikztotarget)$)--($(\tikztotarget)!-#2!(\tikztostart)$)%
  \tikztonodes}}
} 

\newcommand{\comment}[1]{}

\begin{document}
\title[]{Complex conjugation and simplicial algebraic hypersurfaces}

\author[Charles Arnal]{Charles Arnal}

\address{Charles Arnal, Univ. Paris 6, IMJ-PRG, France.}
\email{charles.arnal@imj-prg.fr} 

\maketitle
%\subjclass[2000]{Primary 14P25; Secondary 14P05} %Secondary 14P05
%\keyword

\begin{abstract}
We call a real algebraic hypersurface in $(\C^*)^n$\textit{simplicial} if it is given by a real Laurent polynomial in $n$-variables that has exactly $n+1$ monomials with non-zero coefficients and such that the convex hull in $\R^n$ of the $n+1$ points of $\Z ^n$ corresponding to the exponents is a non-degenerate $n$-dimensional simplex.
Such hypersurfaces are natural building blocks from which more complicated objects can be constructed, for example using O. Viro's Patchworking method. Drawing inspiration from related work by G. Kerr and I. Zharkov \cite{KerrZharkov}, we describe the action of the complex conjugation on the homology of the coamoebas of simplicial real algebraic hypersurfaces, hoping it might prove useful in a variety of problems related to topology of real algebraic varieties. In particular, assuming a reasonable conjecture inspired by \cite{KerrZharkov}, we identify the conditions under which such a hypersurface is Galois maximal.

\end{abstract}

\tableofcontents

%solution brutale pour le bas de page
\newcommand\blfootnote[1]{%
  \begingroup
  \renewcommand\thefootnote{}\footnote{#1}%
  \addtocounter{footnote}{-1}%
  \endgroup
}
\blfootnote{ This research was supported by the DIM Math Innov de la R\'egion Ile-de-France.
\begin{center}
\includegraphics[width = 20mm]{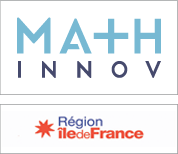} 
\end{center}  }

\section*{Acknowledgement}
The author is very grateful to Ilia Itenberg, Oleg Viro and Ilia Zharkov for helpful discussions.

\input{Introduction.tex}

\input{Definitions.tex}

\input{Coamoebas.tex}

\input{HomologicalComputations.tex}

\input{GaloisMaximality.tex}

\bibliographystyle{alpha}
\bibliography{biblio}

\end{document}

%% file: Introduction.tex
\section{Introduction}

Any non-constant real Laurent polynomial $P$ in $n$ variables defines a real algebraic hypersurface $X_P$ in the complex torus $(\C^*) ^n$. Given such a polynomial $P(z)=\sum_{\lambda = ( \lambda_1,\ldots,\lambda_n) \in \Lambda } c_\lambda z_1^{\lambda_1},\ldots, z_n^{\lambda_n} $, where $\Lambda$ is a finite subset of $\Z ^n$ and $c_\lambda \in \R^*$ for all $\lambda \in \Lambda$, we call the convex hull in $\R^n$ of $\Lambda$ the \textit{Newton polytope} of $P$.

If $P$ has exactly $n+1$ monomials with non-zero coefficients and if the Newton polytope $\Delta _P$ of $P$ is a non-degenerate $n$-dimensional simplex, we call $P$ a \textit{simplicial} real polynomial and $X_P$ a \textit{simplicial} real algebraic hypersurface. In particular, the monomials of $P$ and the vertices of $\Delta_P$ are in bijection.

Let $X$ be a simplicial real algebraic hypersurface.
We denote the real (respectively, complex) points of $X$ in $(\R^*) ^n$ (respectively, $(\C^*)^n$) by $\R X$ (respectively, $\C X$).
The complex conjugation $c$ on the complex torus naturally acts on $\C X$. Our goal here is to better understand this action and the induced action $c_*$ on the homology $H_*(\C X)$ of $\C X$ with coefficients in $\Z_2:=\Z/2\Z$ (unless otherwise specified, we always consider homology with coefficients in $\Z_2$).

Simplicial real algebraic hypersurfaces play a key role in Viro's Patchworking method, in particular in the combinatorial case, where varieties with prescribed topology are obtained by gluing together simplicial real algebraic hypersurfaces. Viro's method is one of the main tools in the study of topology of real algebraic varieties, and is deeply connected to tropical geometry. It was used by Viro to disprove the \textit{Ragsdale conjecture} regarding real algebraic curves in \cite{Viro80}, and later in \cite{Itenberg93} by I. Itenberg to show that the conjecture was also asymptotically wrong.  A complete exposition of the method can be found in \cite{ViroPatchworking} .

Hence, in addition to improving our understanding of simplicial real algebraic hypersurfaces, which are among the most natural and simple examples of real algebraic varieties, our results might also lead to a better comprehension of the varieties obtained using the combinatorial case of Viro's Patchworking.

In particular, A. Renaudineau and K. Shaw proved in \cite{RS} a bound conjectured by Itenberg, which states that in the case of a real algebraic projective hypersurface $X \subset \R \mathbb{P}^n$ obtained via \textit{primitive} Patchworking, 
\begin{equation}\label{BorneKristinArthur1}
    \dim_{\Z _2}(H_q(\R X)) \leq h^{q,n-1-q}(\C X) +1
\end{equation}
if $q\in \{0,\ldots,n\}$ and $q\neq \frac{n-1}{2}$, and
\begin{equation}\label{BorneKristinArthur2}
    \dim_{\Z _2}(H_\frac{n-1}{2}(\R X)) \leq h^{\frac{n-1}{2},\frac{n-1}{2}}(\C X)
\end{equation}
if $n-1$ is even, where $h^{q,p}(\C X)$ is the $(q,p)$-th Hodge number of $\C X$. This can be compared to the much looser \textit{Smith-Thom inequality}, valid for all real algebraic varieties, which states that 
\begin{equation}\label{SmithThom}
    \dim_{\Z _2}(H_*(\R X)) \leq \dim_{\Z _2}( H_*(\C X) ).
\end{equation}
Taking the sum of inequalities (\ref{BorneKristinArthur1}) and (\ref{BorneKristinArthur2}) over all $q\in \{0,\ldots,n\}$ yields the Smith-Thom inequality. A real algebraic variety for which (\ref{SmithThom}) is an equality is called \textit{Smith-Thom maximal}, or simply \textit{maximal}. They are objects of particular interest.

A primitive Patchwork is a Patchwork obtained by gluing together simplicial real algebraic hypersurfaces whose Newton polytopes are of Euclidean volume $\frac{1}{n!}$ (more details can be found in \cite{RS}). Such hypersurfaces are the simplest possible objects from which to construct more complicated hypersurfaces; in particular, they are maximal, which is a key ingredient in the proof of inequalities (\ref{BorneKristinArthur1}) and (\ref{BorneKristinArthur2}).

The \textit{Borel-Swan inequality}, which implies the Smith-Thom inequality,  states that
\begin{equation*}
    \dim_{\Z _2}(H_*(\R X)) \leq \dim_{\Z _2}\left( \frac{Ker(1+c_*)}{Im(1+c_*)} \right),
\end{equation*} where $1+c_* : H_*(\C X) \longrightarrow H_*(\C X)$  (from now on, for any $\Z_2$-vector space $V$, we write indifferently $\dim_{\Z_2}(V)$ or $\dim (V)$ to denote the dimension over $\Z_2$ of $V$). When this inequality is an equality, $X$ is called \textit{Galois maximal} (see  \cite{RealEnrique} for additional details). The notion has been considered of interest in itself; see for example Krasnov \cite{Krasnov}, where the Galois maximality of various families of varieties is proved.

Galois maximality is equivalent to a condition of degeneration on the second page of a certain associated spectral sequence, called \textit {Kalinin spectral sequence}, which was introduced by Kalinin in \cite{Kalinin2} - further explanations can also be found in \cite{StiefelOrientations}. Smith-Thom maximality is equivalent to degeneration on the first page of this spectral sequence,  which plays an important role in the arguments used in \cite{RS}.

Therefore, one possible way to relax the primitivity condition in order to generalize the results from \cite{RS} could be to consider Patchworks obtained from gluing together simplicial real algebraic hypersurfaces that are merely Galois maximal, rather than primitive (and thus Smith-Thom maximal). Note that we already know that inequalities (\ref{BorneKristinArthur1}) and (\ref{BorneKristinArthur2}) do not hold as they are under such hypotheses - additional work is needed.

 Hence our interest in the action of the conjugation on the homology of simplicial real algebraic hypersurfaces, and in particular in the rank of the map $1+c_*$, in order to determine the conditions under which such a hypersurface is Galois maximal.

In  \cite{KerrZharkov}, G. Kerr and I. Zharkov have in particular showed that the complex part $\C X$ of a simplicial real algebraic hypersurface is homeomorphic to the associated \textit{phase tropical variety} $\mathcal{T}X$.
Moreover, it is easy to see that the phase tropical variety $\mathcal{T}X$ retracts by deformation to the \textit{coamoeba} $\CC_X$ of $X $. Hence we get a homotopy equivalence $\C X \longrightarrow \CC _X$.
Private conversations with Zharkov have led us to believe that this homotopy equivalence satisfies the condition of the following conjecture (which we hope to prove soon):

\begin{conj}\label{ConjectureConjugationZharkov}
Let $X$ be a simplicial real algebraic hypersurface. There is a homotopy equivalence $\phi: \C X \longrightarrow \CC_X$ such that the following diagram commutes for all $i\geq 0$:
\begin{center}
\begin{tikzcd}
H_i(\C X) \arrow[r, "\phi_*"] \arrow[d, "c_*"]
& H_i(\CC_X) \arrow[d, "c_*"] \\
H_i(\C X) \arrow[r, "\phi_*"]
& H_i(\CC_X)
\end{tikzcd}
\end{center}
where $c_*$ is induced by the complex conjugation on either $\C X$ or $\CC_X$.
\end{conj}
This would immediately imply
\begin{align*}
\dim (Im(1+c_* :H_{i}(\C X)\rightarrow H_{i}(\C X) )) =\\ \dim (Im(1+c_* :H_{i}(\CC_ X)\rightarrow H_{i}(\CC_ X) )),
\end{align*}
where $1$ is the identity.

Our main focus in the rest of this text is to prove the following result, though the details of the proof might be of interest in themselves when considering other related questions.
Consider a simplicial real algebraic hypersurface $X$, the Newton polytope $\Delta$ of the simplicial real Laurent polynomial $P$ that defines $X$, and pick a vertex $O$ of $\Delta$. The edges of $\Delta$ containing $O$ define $n$ integer vectors (choosing $O$ as their initial point). Define $A\in M_{n\times n}(\Z_2)$ as the matrix whose rows are these $n$ vectors modulo $2$. Then:

\begin{thm}\label{MainResult}
If $\R X$ intersects non-trivially each quadrant of $(\R^*)^n$, then 
$$\dim_{\Z _2}(H_*(\R X))  =  \dim_{\Z _2}\left( \frac{Ker(1+c_*)}{Im(1+c_*)} \right),$$
 where $1+c_* : H_*(\CC_X) \longrightarrow H_*(\CC_ X)$.

Otherwise, we have
\begin{align*}
\dim_{\Z _2} \left( \frac{Ker(1+c_*)}{Im(1+c_*)}\right) - \dim_{\Z _2}( H_*(\R X))  = 2(2^{n-\rank _{\Z_2}(A)}-1- (n-\rank _{\Z_2}(A))).
\end{align*}
\end{thm}
The condition on the intersection of $\R X$ with the quadrants of  $(\R^*)^n$ is equivalent to a condition on the matrix $A$ and the signs of the monomials of $P$, as proved in Lemma \ref{LemmaRealPartHomology}.

If Conjecture \ref{ConjectureConjugationZharkov} holds, then Theorem \ref{MainResult} immediately implies 
\begin{thm}\label{TheoremConjugationMainResultConsequence}
If $\R X$ intersects non-trivially each quadrant of $(\R^*)^n$, then $X$ is always Galois maximal.

Otherwise, we have
\begin{align*}
\dim_{\Z _2} \left( \frac{Ker(1+c_*)}{Im(1+c_*)}\right) - \dim_{\Z _2}( H_*(\R X))  = 2(2^{n-\rank _{\Z_2}(A)}-1- (n-\rank _{\Z_2}(A))),
\end{align*}
where $1+c_* : H_*(\C X) \longrightarrow H_*(\C X)$,
and $X$ is Galois maximal if and only if $n-\rank _{\Z_2}(A) \in \{0,1\}$.
\end{thm}

The rest of this article is organized as follows:
in Section \ref{Definitions}, we go over some definitions and notations.
In Section \ref{Coamoebas}, the coamoeba $\CC_X$ associated to $X$ is introduced and described in a way suited to computations. As the coamoeba $\CC_X$ is homotopically equivalent to $\C X$, we can study the action of the conjugation on $\CC _ X$ and deduce from it its action on $\C X$ (assuming Conjecture \ref{ConjectureConjugationZharkov}).
In Section \ref{HomologicalComputations}, the action of the conjugation on the homology of $\CC_ X$ is described. In particular, the rank of $1+c_ *$ is computed.
Finally, Theorem \ref{MainResult} is proved in Section \ref{GaloisMaximality}.

%% file: Definitions.tex
\section{Definitions and notations}\label{Definitions}

We denote the $n$-dimensional torus $(\R  /2\pi \Z  )^n$ (seen as the product of $n$ unit circles) by $T^n$, and use either additive or multiplicative notations, based on context, to describe its natural group law. In particular, we frequently apply matrices with integer coefficients to elements of $T^n$.

Given a vector of signs $(\varepsilon_1, \ldots,\varepsilon_n) \in \{1,-1 \}^n $, we define $\delta(\varepsilon)=(\delta_1, \ldots,\delta_n)\in \Z_2^n$ by the relation $\varepsilon_i = (-1) ^{\delta_i}$.\\
Given a finite set $S$, let $|S|$ denote the cardinality of $S$. \\
We define
$Arg: (\C^*)^n \longrightarrow T^n,\quad (z_1, \ldots, z_n) \longmapsto (arg(z_1),\ldots,arg(z_n)).$

Throughout this text, we use the following conventions: for any $z=(z_1, \ldots, z_n) \in \C^n$, any matrix $G\in M_{n\times n}(\Z) $ and any vector $v=(v_1,\ldots, v_n) \in \Z^n$, we define $z^v:= z_1^{v_1} z_2^{v_2} \ldots z_n^{v_n} \in \C$ and $z^G := (z^{G^1}, \ldots, z^{G^n}) \in \C^n$, where $G^i$ is the $i$-th \textbf{row} of $G$. Choosing rows instead of columns has the advantage of allowing us to write $Arg(z^G) = G\cdot Arg(z)$, and the disadvantage that for another matrix $H\in M_{n\times n}(\Z)$, we have ${(z^G)}^H=z^{H\cdot G}$.

Consider as above a simplicial polynomial $P(z) =\sum _{\alpha\in \Delta_P} c_\alpha z^\alpha$, for some coefficients $c_\alpha \in \R^*$, and the associated simplicial real algebraic hypersurface $X_P$. Up to multiplication by a non-trivial Laurent monomial (which doesn't change $X_P$), the polynomial $P$ can be chosen so that $\Delta_P$ has $0\in \Z ^n$ as one of its vertices. We assume this to be the case from now on. Moreover, let us denote the non-null vertices of $\Delta_P$ by $\alpha^1_P,\ldots,\alpha^n_P \in \Z ^n$, and define $A_P \in M_{n\times n}(\Z)$ as the matrix whose $i$-th \textbf{row} is $\alpha^i_P$ for $i=1, \ldots ,n$. If we also define $\alpha^0_P:=0\in \Z ^n$, we can write indifferently $P(z) =\sum _{\alpha\in \Delta_P} c_\alpha z^\alpha = \sum _{i=0,\ldots,n} c_{\alpha_P^i} z^{\alpha_P^i}  = \sum _{i=1,\ldots,n} c_{\alpha_P^i} z^{A_P^i}+c_0.$

For any $G\in M_{n\times n}(\Z)$, we define the algebraic morphism 
\begin{align*}
   \Phi _G : \quad  (\C^*)^{n}  & \to  (\C^*)^{n} \\
    z & \mapsto  z^G.
\end{align*}

If $G$ is invertible, $\Phi_G$ is an algebraic isomorphism that sends $X_{\tilde{P}} = \{ \tilde{P}(z)=\sum _{i=1,\ldots,n} c_{\alpha_P^i} z^{(A_P \cdot G)^i}+c_0= 0  \}$ to $X_P = \{ P(z)=\sum _{i=1,\ldots,n} c_{\alpha_P^i} z^{A_P^i}+c_0= 0  \}$ and respects the real structure. This means that up to such isomorphisms, we can consider the matrix $A_P$ defining the simplex $\Delta_P$ up to right-multiplication by invertible matrices with integer coefficients.  

\begin{remark}\label{Smoothness}
The complex part $\C X_P \subset (\C^*)^n$ of a simplicial real algebraic hypersurface $X_P$ is smooth.
\end{remark}
Let $X_P$ be given as above by $P(z) = \sum _{i=1,\ldots,n}  c_{i} z^{\alpha_i} +c_0$, for some coefficients $c_i \in \R^*$ and a simplex $\Delta_P$. Let $A$ be the $n \times n$ matrix whose rows correspond to the vertices $\alpha_i$ of $\Delta _ P$ and the monomials of $P$ (for $i=1,\ldots,n$). We also introduce $B$, the cofactor matrix of $A$, and $P_L(z) = \sum _{i=1,\ldots,n}  c_{i} z_i + c_0$ (the $L$ stands for "linear").

Then $P(z)=P_L\circ \Phi_A(z)$, thus $\Phi_A:(\C^*)^{n}   \to  (\C^*)^{n}$ maps $\C X_P$ to the hyperplane $\C X_{P_L} :=\{P_L(z) = 0 \}$. But $\Phi_A \circ \Phi_B = \Phi_{det(A) \cdot Id}: (z_1,\ldots,z_n) \mapsto (z_1^{det(A)},\ldots,z_n ^{det(A)})$ is a local diffeomorphism, which implies that so is $\Phi_A$ by a dimensional argument. Hence the smoothness of $\C X_P$ can be deduced from the smoothness of $\C X_{P_L} $.

%% file: Coamoebas.tex
\section{Coamoebas}\label{Coamoebas}
As above and for the remainder of this article, let $P(z)=\sum _{i=1,\ldots,n} \varepsilon_i c_{i} z^{\alpha^i} + c_0$, where $\varepsilon=(\varepsilon_1, \ldots, \varepsilon_n) \in \{1,-1 \}^n$ and $c_i \in \R_{>0}$  for all $i$,  be a simplicial real polynomial (we can suppose without loss of generality that the constant term is positive). Let $X:=X_P \subset (\C^*)^n$ be the associated simplicial real algebraic hypersurface, and let $\Delta_P$ be the Newton polytope of $P$, and $A$ be the $n \times n$ matrix whose rows correspond to the vertices of $\Delta _ P$ and the monomials of $P$, \textit{i.e.} $A^i = \alpha^i$.

\subsection{Definition and description of $\CC_X$}

The \textit{coamoeba} $\CC_X \subset T^n$ of $X$ is the closure in $T^n$ of the image $Arg(\C X)$. The conjugation $c$ acts as $-\id $ on $T^n$; if we fix a representation of $T^n$ as $[0,2\pi]^n$ with its boundary quotiented as needed, which we do from now on, $c$ is the central symmetry.

%In  \cite{KerrZharkov} (Theorem 21), Kerr and Zharkov proved that $X$ is homeomorphic to the associated \textit{phase tropical hypersurface} $\mathcal{T}X \subset \R ^n \times T^n$ .

%It can be seen from the same article that there is a deformation retraction of $((\C^*)^n, \mathcal{T}X)$  to $(T^n,\CC_X)$ that is compatible with the conjugation. Hence, we have a homotopy equivalence of triples between $((\C ^*) ^n,\C X, c)$ and $(T^n,\CC_X, c)$. We study the latter triple.

%VOIR AVEC ZHARKOV SI COMPLETEMENT JUSTE POUR LES TRIPLES, S INON OUBLIER L'ESPACE AMBIANT

We will use the convenient description of $\CC_X$ given by Kerr and Zharkov in the linear case, \textit{i.e.} when the coamoeba is given by an affine hyperplane, as an intermediate step to get to the general case.
Let us introduce the simplicial polynomials $P_L^+(z)= \sum _{i=1,\ldots,n} c_{i} z_i + c_0$ and $P_L(z)=\sum _{i=1,\ldots,n}  \varepsilon_i c_{i} z_i +  c_0$, where the coefficients $c_i$ and $\varepsilon_i$ are the same as in $P$. Name $\CC_{X_{P_L^+}}$ and $\CC_{X_{P_L}}$ the associated coamoebas.

Consider $T:=T^1=\R /2\pi \Z $ and identify it with the unit circle via the map $E:[\theta] \mapsto \exp(i\theta)$. 
We say that  $(\theta _1, \ldots, \theta _ n) \in T^n$ are in an \textit{allowed configuration} if there is no open half-circle in $T$ containing all the $\theta_i$ as well as the point $1=E[0]$ ( which corresponds to the constant term).

 A \textit{zonotope} is the Minkowski sum in $\R^n$ of a finite collection of segments. Now consider such a zonotope $\tilde{Z}$; if the quotient map $\R^n \longrightarrow T^n= (\R / 2\pi \Z )^n$ restricted to the interior of $\tilde{Z}$ is an embedding, we call in the rest of this article,  by extension, the image $Z$ of the interior of $\tilde{Z}$ an (open) zonotope of $T^n$.
 
 \begin{lemma}\label{CoamoebaLinearPlusCase}
 The points of $\CC_{X_{P_L^+}}$ are exactly the n-tuples $(\theta _1, \ldots, \theta _ n) \in T^n$ in allowed configurations. There is an $n$-dimensional zonotope $ \tilde{Z}$ in $\R^n$, generated by $n+1$ segments, such that the quotient map $\R^n \longrightarrow T^n= (\R / 2\pi \Z )^n$ restricted to the interior of $\tilde{Z}$ is an embedding, and such that the open zonotope $Z$ of $T^n$ which is the image of the interior of $\tilde{Z}$ is the complement of the points in allowed configurations (see Figure \ref{Zonotope}).
 \end{lemma}
 
 The proof is almost trivial - see \cite{KerrZharkov} for related details. Note that $Z$ is contractible in $T^n$, hence $(T^n,\CC_{X_{P_L^+}})$ is of the same homotopy type as $( T^n,T^n \backslash \{ \star \})$.

\begin{figure}
\includegraphics[scale=0.2]{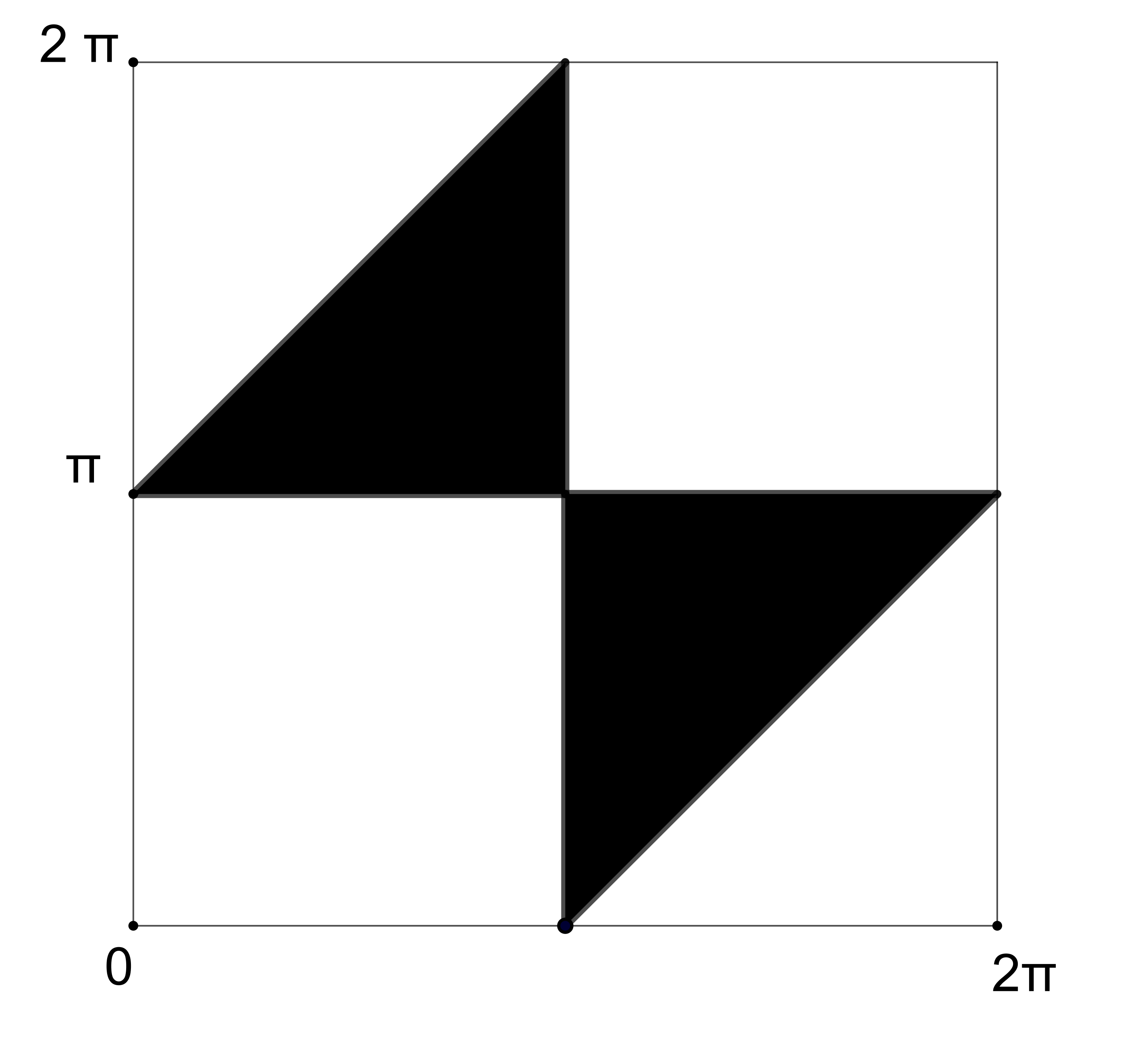}
\caption{
In black, the coamoeba $\CC _{X _{P_L^+}}$. In white, the open zonotope $Z$}
\label{Zonotope}
\end{figure}

We now consider the slightly more general case of $\CC_{X_{P_L}}$. Let $\delta=(\delta_1, \ldots,\delta_n):=\delta(\varepsilon_1, \ldots, \varepsilon_n)\in \Z_2^n$ be as in Section \ref{Definitions}.

\begin{lemma}\label{CoamoebaLinearCase}
The coamoeba $\CC_{X_{P_L}}$ is a translate of $\CC_{X_{P_L^+}}$ by $\pi \cdot \delta \in T^n$, \textit{i.e.} $\CC_{X_{P_L}}=\CC_{X_{P_L^+}}+\pi \cdot \delta \subset T^n$.
\end{lemma}

The zonotope of forbidden configurations from Lemma \ref{CoamoebaLinearPlusCase} has simply been translated; the proof is once again almost trivial.

Let us now consider the general case of $\CC_X$.

\begin{lemma}\label{CoamoebaGeneralCase}
Using the same notations as above, the coamoeba $\CC _X $ is the preimage of the coamoeba $\CC_{X_{P_L}}$ by the map $A \cdot : T^n \longrightarrow T^n$, \textit{i.e.} $ \CC_X = A^{-1}(\CC_{X_{P_L}})$. 
\end{lemma}
\begin{proof}
First, let us show that $Arg(\C X) = A^{-1}(Arg(\C X_{P_L}))$.

Indeed, if $\theta =(\theta_1, \ldots, \theta_n) \in Arg(\C X)$, by definition there exists $r=(r_1, \ldots, r_n) \in (\R _{>0})^n$ such that  $z:= (r_1\theta_1,\ldots, r_n \theta _n)$ belongs to $\C X$, \textit{i.e.} $P(r_1\theta_1,\ldots, r_n \theta _n)=\sum _{i=1,\ldots,n} \varepsilon_i c_{i} r^{A^i}\theta^{A^i}+ c_0= P_L(r^{A^1}\theta^{A^1}, \ldots, r^{A^n}\theta^{A^n})=0$. Hence $(r^{A^1}\theta^{A^1}, \ldots, r^{A^n}\theta^{A^n})\in \C X_{P_L}$ and $\theta ^A = A \cdot \theta$ belongs to $Arg(\C X_{P_L})$.

Conversely, suppose that $A \cdot \theta \in Arg(\C X_{P_L}).$ Then by definition there exists $s=(s_1, \ldots, s_n) \in (\R _{>0}) ^n$ such that $P_L(s_1 \theta^{A_1}, \ldots, s_n \theta^{A_n})=\sum _{i=1,\ldots,n} \varepsilon_i c_{i} s_i \theta^{A^i}+ c_0= 0$. If there exists $r=(r_1, \ldots, r_n) \in (\R _{>0})^n$ such that $r^A=s$, then $z:=(r_1\theta_1,\ldots, r_n \theta _n)$ is such that $P(z) = 0$ and we can conclude that $\theta \in Arg(\C X)$.
Now consider the cofactor matrix $B$ of $A$, and 
\begin{align*}
     (\R_{>0})^{n}  & \xrightarrow{\Phi_B} (\R_{>0})^{n}   \xrightarrow{\Phi_A}  (\R_{>0})^{n} \\
     r =(r_1, \ldots, r_n) & \longmapsto r^B  \longmapsto  (r^B)^A=r^{det(A)} = (r_1^{det(A)}, \ldots, r_n^{det(A)}).
\end{align*}
This clearly shows that $\Phi_A$ is surjective from $ (\R_{>0})^{n} $ onto itself; thence we have shown that  $Arg(\C X) = A^{-1}(Arg(\C X_{P_L}))$.

As $A \cdot : T^n \longrightarrow T^n$ is continuous, we immediately have that $\CC _X = \overline{Arg(\C X)} \subset A^{-1}(\overline{Arg(\C X_{P_L})}) =A^{-1}(\CC_{X_{P_L}})$.
Following the same reasoning as in Remark \ref{Smoothness}, the map  $A \cdot : T^n \longrightarrow T^n$ is a local diffeomorphism, because $(AB) \cdot = (det(A) \cdot Id)\cdot  : T^n \longrightarrow T^n$ is a local diffeomorphism, where $B$ is the comatrix of $A$.
Consider $\theta \in A^{-1}(\overline{Arg(\C X_{P_L})})$, an open neighborhood $U \subset T^n$ of $\theta$ such that $(A \cdot)|_U$ is an embedding, and some open neighborhood $V \subset T^n$ of $\theta$. Then $A\cdot(U \cap V) $ is an open neighborhood of $A\cdot \theta \in \overline{Arg(\C X_{P_L})}$, hence there exists $\rho \in A\cdot(U \cap V)\cap  Arg(\C X_{P_L})$. Then $((A \cdot)|_U)^{-1}(\rho) \in V \cap  A^{-1}(Arg(\C X_{P_L})) = V \cap Arg(\C X) $, which shows that $\theta \in \overline{Arg(\C X)} $.
\end{proof}

\subsection{A more explicit description of $\CC_X$}\label{explicitDescription}

It is well-known (for example, using the theorem of structure of finitely generated abelian groups) that there exists two (non-uniquely defined) invertible matrices $G, H \in M_{n\times n}(\Z)$ such that
$$G\cdot A\cdot H=D=\begin{bmatrix}
    d_{1} & & 0 \\
    & \ddots & \\
    0 & & d_{n}
  \end{bmatrix}, $$
where $d_1|d_2|\ldots|d_n\in \Z$ and $d_1 d_2\ldots d_k$ is the greatest common divisor of the non-trivial $k$-minors of $A$ (for $k=1, \ldots,n$).

Consider $\delta^G = G\cdot \delta \in \Z_2^n$, where $\delta= \delta(\varepsilon_1, \ldots, \varepsilon_n) \in \Z_2^n$ is as above.

We partition $\{1,\ldots,n\}$ in the following way:
$$I^{0,0}:=\{i \in \{1,\ldots,n\} \quad | \quad \delta^G_i \equiv [0]_2, d_i \equiv [0]_2 \},$$
$$I^{1,0}:=\{i \in \{1,\ldots,n\} \quad | \quad \delta^G_i \equiv [1]_2, d_i \equiv [0]_2\},$$
$$I^{0,1}:=\{i \in \{1,\ldots,n\} \quad | \quad \delta^G_i \equiv [0]_2, d_i \equiv [1]_2 \},$$
$$I^{1,1}:=\{i \in \{1,\ldots,n\} \quad | \quad \delta^G_i \equiv [1]_2, d_i \equiv [1]_2\}.$$

%Note that as $d_1|d_2|\ldots|d_n$, for any $i\in I^{0,0} \cup I^{1,0}$ and any $j \in I^{0,1} \cup I^{1,1}$, we have $i>j$.

For the remainder of this text, we fix two such matrices $G$ and $H$. Moreover, as observed in section \ref{Definitions}, we can consider $A$ up to right-multiplication by invertible matrices with integer coefficients: hence we can, and do, assume that $H$ is the identity matrix.
Hence from now on we have

$$
A = G^{-1}\cdot D.
$$

Then $\CC_X$ lends itself to the following description: define the set of indices $\Omega := \{0,\ldots,d_1 -1\} \times \ldots \times \{0,\ldots, d_n -1 \}$ and let $\widetilde{\delta^G} \in \{0,1\}^n$ be the unique lifting of $\delta^G$. Let also $\widetilde{G\cdot \CC_{X_{P_L}}}$ be the preimage of $G\cdot \CC_{X_{P_L}}$ by the quotient map $[0,2\pi]^n \longrightarrow T^n$ (see the left part of Figure \ref{FigureGeneralCoamoeba}).

Cover  $[0,2\pi]^n$ with $d_1 \ldots d_n$ hyperrectangular cells $C_\alpha :=[\alpha_1\frac{2\pi}{d_1}, (\alpha_1 +1)\frac{2\pi}{d_1} ]\times \ldots \times [\alpha_n\frac{2\pi}{d_n},(\alpha_n+1)\frac{2\pi}{d_n} ]$ for $\alpha=(\alpha_1, \ldots, \alpha_n) \in \Omega$. Define $\widetilde{\CC_X} \subset [0,2\pi]^n$ as the set such that the pair $(C_\alpha, C_\alpha \cap \widetilde{\CC_X})$ is mapped to the pair $([0,2\pi]^n ,\widetilde{G\cdot \CC_{X_{P_L}}})$ by  $(z_1, \ldots,z_n) \mapsto (d_1z_1 - \alpha_1 2\pi,\ldots, d_n z_n - \alpha_n 2\pi )$ (see the right part of Figure \ref{FigureGeneralCoamoeba}). Then

\begin{proposition}\label{DescriptionCoamoeba}
The pair $(T^n, \CC _  X)$ is the image of $([0,2\pi]^n ,\tilde{\CC_X})$ by the quotient map $[0,2\pi]^n \longrightarrow T^n$.

The coamoeba $\CC_X$ is the complement in $T^n$ of $d_1 \ldots d_n$ open zonotopes indexed by the set of indices $\Omega$, such that the center of zonotope $Z_\alpha$, for $\alpha= (\alpha_1,\ldots, \alpha_n)\in  \Omega)$,  is in $\pi \cdot (\widetilde{\delta^G}_1 / d_1,\ldots, \widetilde{\delta^G}_n /d_n)+2\pi \cdot (\alpha_1/d_1, \ldots, \alpha_n/d_n) \in T^n$. Those zonotopes are translate of each other.
\end{proposition}

\begin{proof}
From \ref{CoamoebaGeneralCase}, we know that $ \CC _  X= A^{-1}(\CC_{ X_{P_L}}) = D^{-1}(G\cdot \CC_{ X_{P_L}})$.

As explained in Lemmas \ref{CoamoebaLinearPlusCase} and \ref{CoamoebaLinearCase}, the coamoeba $\CC_{X_{P_L^+}}$ is the complement in $T^n$ of an open zonotope centered in $0$, and $\CC_{X_{P_L}}$ is the complement in $T^n$ of the same zonotope translated and now centered in $(\delta_1 \pi, \ldots, \delta _ n \pi)=\pi \cdot \delta \in T^n$.

The linear isomorphism $G$ is then applied to it, so that $ G\cdot \CC_{ X_{P_L}}$ is the complement in $T^n$ of an open zonotope (geometrically the starting zonotope deformed by $G$) centered in $\pi\cdot \delta^G \in T^n$. The last step is to take the preimage by the covering map $D\cdot : T^n \longrightarrow T^n$.

\end{proof}

\begin{figure}
\includegraphics[scale=0.5]{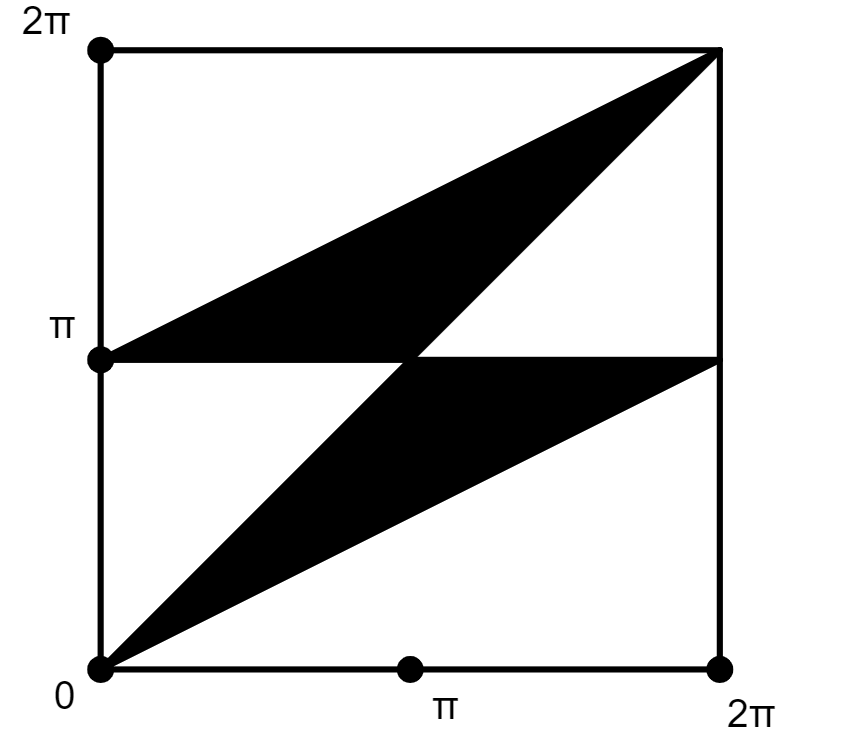}
\includegraphics[scale=0.5]{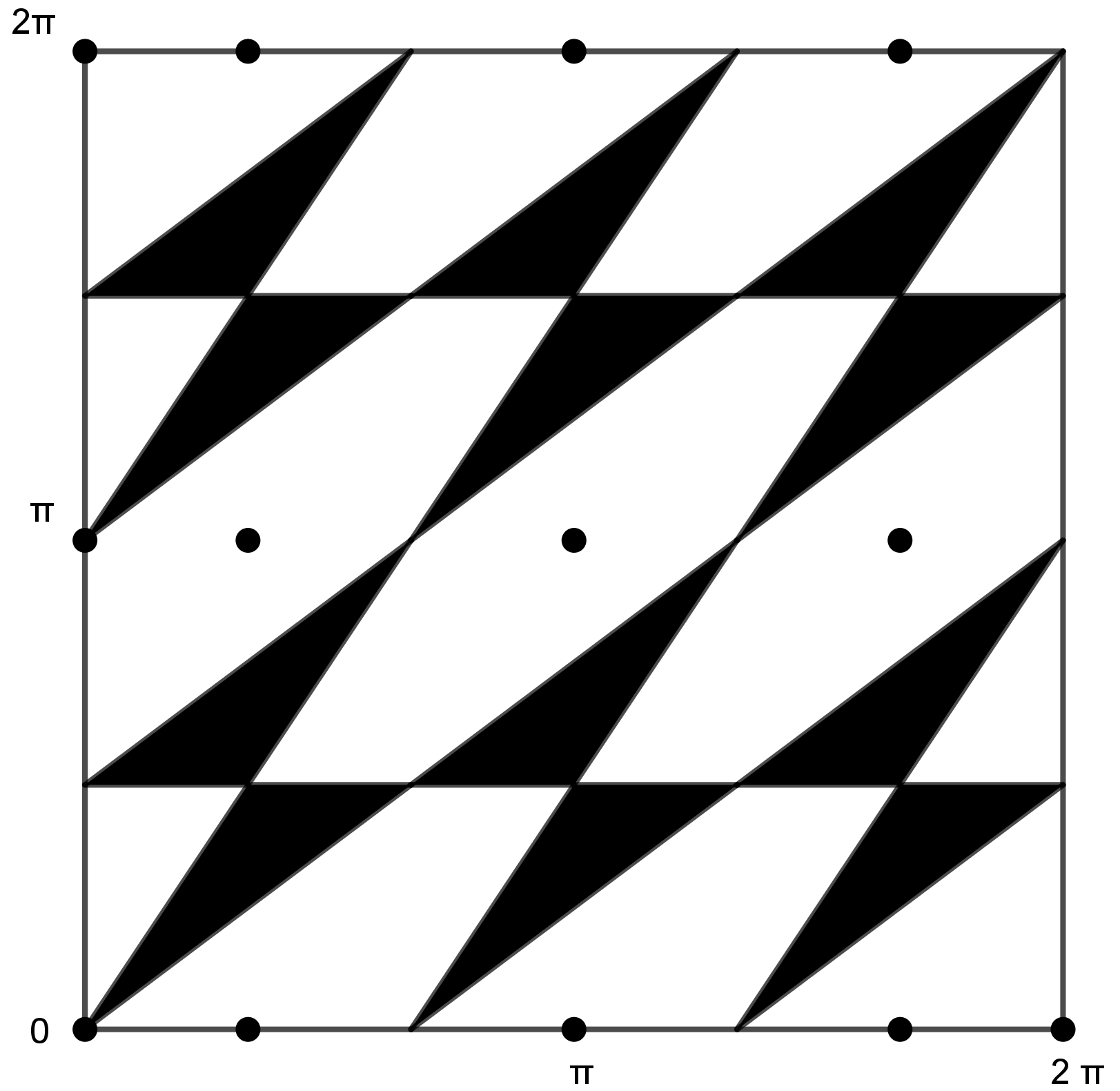}
\caption{
On the left, $\widetilde{G\cdot \CC_{X_{P_L}}}$. On the right:  the coamoeba $\CC _X$ in black and the zonotopes in white.}
\label{FigureGeneralCoamoeba}
\end{figure}

%% file: HomologicalComputations.tex
\section{Homological computations}\label{HomologicalComputations}
Let us now study the action of the conjugation $c$ on $\CC_ X$, using the description from Proposition \ref{DescriptionCoamoeba}. As pointed out earlier, $c$ simply acts as the central symmetry on $T^n$ seen as an appropriate quotient of $[0,2\pi]^n$. It maps $\CC_ X$ to itself, and zonotopes to one another. Let us describe that action more precisely.

\subsection{Action of the conjugation on the zonotopes}\label{SubsectionActionIndices}

Let $\alpha =(\alpha_1, \ldots, \alpha_n) \in \Omega =  \{0,\ldots,d_1 -1\} \times \ldots \times \{0,\ldots, d_n -1\}$. We define $c(\alpha) \in \Omega$ as the index such that $Z_{c(\alpha)}=c(Z_\alpha)$. Then:

If $i\in I^{0,0}$, then $c(\alpha)_i$ is the unique lifting to $\{0,\ldots,d_i -1 \}$ of $[d_i - \alpha_i]_{d_i}$. In particular, notice that $c(\alpha)_i = \alpha_i$ if and only if $\alpha_i \in \{ 0, d_i / 2\}$.

If $i\in I^{0,1}$ , then $c(\alpha)_i$ is the unique lifting to $\{0,\ldots,d_i -1 \}$ of $[d_i - \alpha_i]_{d_i}$. In particular, notice that $c(\alpha)_i = \alpha_i$ if and only if $\alpha_i=0$.

If $i\in I^{1,0}$, then $c(\alpha)_i = (d_i -1)- \alpha_i$. In particular, notice that $c(\alpha)_i$ is never equal to $\alpha_i$.

If $i\in I^{1,1}$, then $c(\alpha)_i = (d_i -1)- \alpha_i$. In particular, notice that $c(\alpha)_i = \alpha_i$ if and only if $\alpha_i=(d_i-1)/2$.

Denote by $F:=\{\alpha \in \Omega | c(\alpha) = \alpha \}$ the set of fixed points of $c$. If $|I^{1,0}| \neq 0$, then $F=\emptyset$. If $|I^{1,0}| = 0$, there are $2^{|I^{0,0}|}$ fixed points: the indices $\beta \in \Omega$ such that $\beta_i \in \{ 0, d_i / 2\}$ if $i\in I^{0,0}$, $\beta_i=0$ if $i\in I^{0,1}$ and $\beta_i=(d_i-1)/2$ if $i\in I^{1,1}$.

\subsection{Homology of $\CC_X$}\label{SubsectionHomologyCoamoeba}

Denote by $\Gamma_1, \ldots, \Gamma_n \in H_{n-1}(T^n)$ the $n$ homology classes induced by the $(n-1)$-dimensional tori $\{[\theta_i] =0\} \subset T^n $. They form a basis of $H_{n-1}(T^n)$.

Using a Mayer-Vietoris long exact sequence, it is easy to see that 
$$ H_k(\CC_X) \xrightarrow{in_*} H_k(T^n),$$
where $in_ *$ is induced by the inclusion, is an isomorphism for $k=0,\ldots,n-2$, and that $ H_k(\CC_X)=0$ for $k\geq n$.
We also get a short exact sequence 
$$ 0 \longrightarrow H_n(T^n) \longrightarrow H_{n-1}(\bigsqcup_{\alpha\in \Omega} \partial Z_\alpha) \longrightarrow H_{n-1}(\CC _X) \longrightarrow H_{n-1}(T^n) \longrightarrow 0 ,$$
where  $\partial Z_ \alpha$ is the border of zonotope $Z_\alpha$, which is homeomorphic to a $(n-1)$-sphere. From it, we deduce the following Lemma:

\begin{lemma}\label{descriptionHomologie}
Given $B_1, \ldots B_n \in H_{n-1}(\CC _X)$ such that $in_*(B_i) =\Gamma_ i $ for all $i$, we have an isomorphism
\begin{equation*}
    H_{n-1}(\CC _X) \cong \left(\bigoplus_{i=1,\ldots,n} \Z_2 \cdot B_i \right) \oplus \left( \frac{\bigoplus_{\alpha\in \Omega} \Z_2 \cdot [\partial Z_\alpha]}{\Z_2 \cdot \sum_{\alpha\in \Omega} [\partial Z_\alpha]}\right).
\end{equation*} 
\end{lemma}

We already know that $c_*([\partial Z_ \alpha]) = [\partial Z_ {c(\alpha)}]$ (as we consider homology with coefficients in $\Z_2$, we need not concern ourselves with orientation), where $c(\alpha)$ is as above - we only need to find suitable classes $B_i$ and describe how $c_*$ acts on them.

Let us define $B_i$, for $i\in \{ 1,\ldots,n\}$. Consider the $(n-1)$-dimensional torus $\{[\theta_i] =0\} \subset T^n $, and the set $I_i$ of intersections between $\{[\theta_i] =0\}$ and the open zonotopes of the complement of $\CC _X$ in $T^n$. Notice that a given zonotope $Z_\alpha$ can intersect several times $\{[\theta_i] =0\}$ - each of these intersections appears as a distinct element of $I_i$. Call $N_i(\alpha)$ the number of such intersections. We would like to define $B_i$ as the class of $\{[\theta_i] =0\}$, with a modification (since $\{[\theta_i] =0\}$ is in general not included in $\CC_X$).

For each intersection $\gamma \in I_i$, let $\alpha(\gamma)\subset \Omega$ be the index of the zonotope corresponding to $\gamma$. The intersection $\gamma$ of $\{[\theta_i] =0\}$ and $\overline{Z_{\alpha(\gamma)}}$ cuts $\partial Z_{\alpha(\gamma)}$ in two (topological) half-spheres of dimension $n-1$. Call $S_\gamma^+$ the half-sphere that lies in the direction $e_i$ of $\{[\theta_i] =0\}$, and $S_\gamma^-$ the one lying in direction $-e_i$ .

We start our construction with $\{[\theta_i] =0\}$. For each $\gamma \in I_i$, we remove the intersection $\{[\theta_i] =0\}\bigcap Z_{\alpha(\gamma)}$, and we add $S_\gamma^+$. Thus we obtain a $(n-1)$-cycle - name it $\widetilde{B_i}$ and name its class $B_i$.
Observe that parts of the border of a given zonotope $Z_\alpha$ can appear several times in $\widetilde{B_i}$: in other words, there can be several $S_\gamma^+$ that are subsets of the same $\overline{Z_\alpha}$ (with non-trivial intersections).
Notice also that $S_\gamma^+$ can be homotopically contracted in $T^n$ to the intersection $\{[\theta_i] =0\}\bigcap \overline{Z_{\alpha(\gamma)}}$. Thus $B_i$ is a lift of $\Gamma_i$, as required. This construction is illustrated in Figure \ref{CycleConstruction}.

\begin{figure}
\includegraphics[scale=0.6]{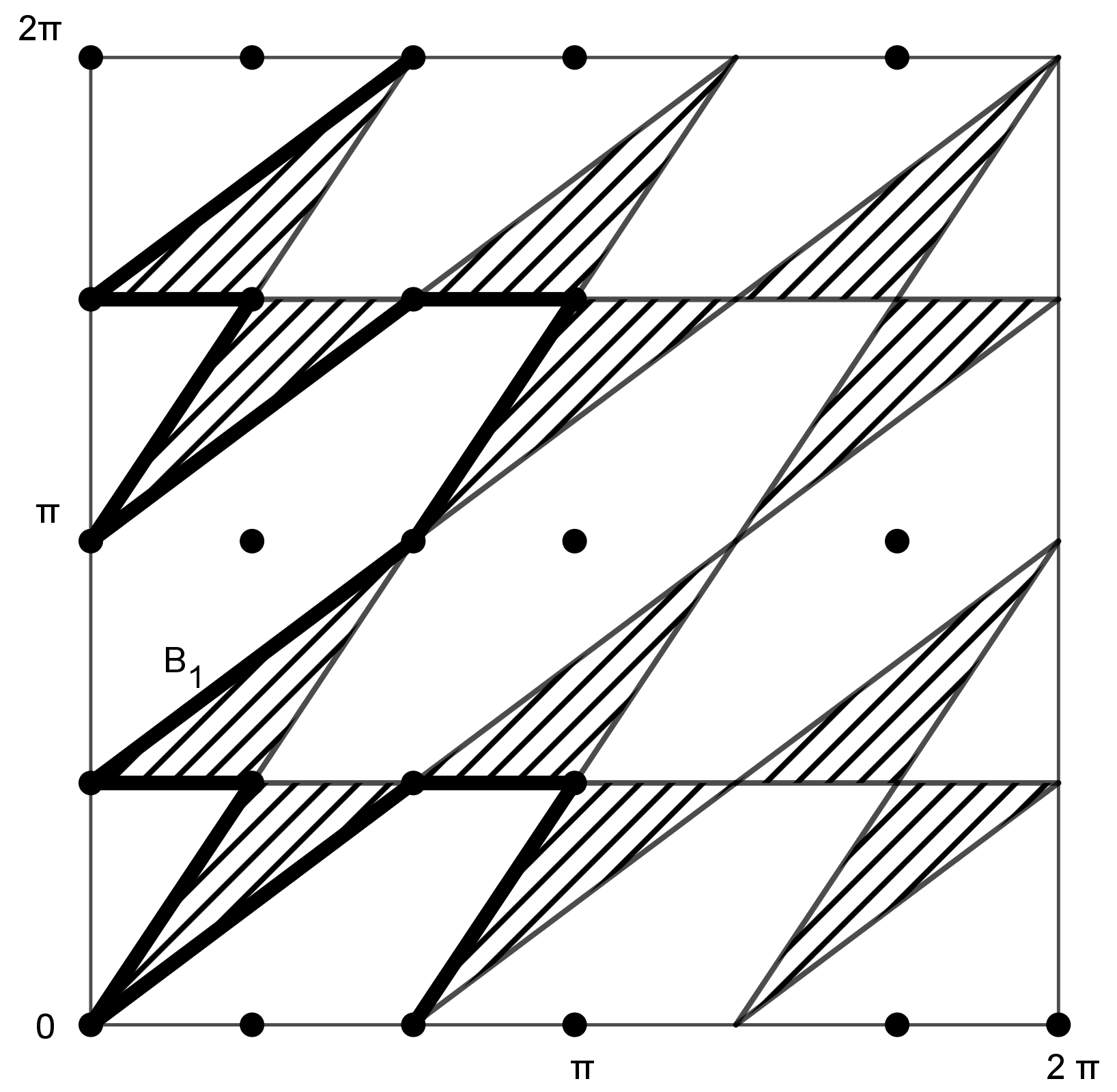}
\caption{
As a black broken line, the cycle $\widetilde{B_1}$. In that example, $I_1$ is of cardinal $4$.}
\label{CycleConstruction}
\end{figure}

\subsection{Image of $1+c_*$}

We want to compute the image of $1+c_*:H_{n-1}(\CC _ X) \longrightarrow H_{n-1}(\CC _ X)$, and in particular its dimension. Using the description in Lemma \ref{descriptionHomologie}, we see that  
\begin{equation}\label{descriptionImage}
   Im(1+c_*)= \sum_{i=1,\ldots,n}\Z_2 \cdot ((1+c_*)(B_i)) + \sum_{\alpha \in \Omega} \Z_2 \cdot [\partial Z_\alpha + \partial Z_{c(\alpha)}]
\end{equation}

Let us make some observations about the image of $B_i$ under $c_*$. It is the class of the mirror image by central symmetry of $\widetilde{B_i}$. This means that it consists of $\{[\theta_i] =0\}$, without the intersections $\{[\theta_i] =0\}\bigcap Z_{\alpha(\gamma)}$ (for $\gamma \in I_i$), and with all the $S_\gamma^-$, i.e. the half-spheres lying in the direction $-e_i$.
Thus $B_i+c_*(B_i) = \sum_{\gamma \in I_i}[\partial Z_{\alpha (\gamma)}]$ (since by definition $S_\gamma^+ +S_\gamma^- = \partial Z_\alpha$) - each $\partial Z_{\alpha}$ appears $N_i(\alpha)$ times.

It would be possible to compute $N_i(\alpha)$ exactly. However, we only want to compute the image of $1+c_*$ - the following observations suffice for our purpose.

Since we are considering coefficients in $\Z _2$, we only have to pay attention to the parity of $N_i(\alpha)$.
Let us define $J_i : =\{ \alpha \in \Omega \quad | \quad N_i(\alpha) \equiv [1]_2\}$. The number of intersections $N_i(\alpha)$ depends only on the length of the zonotope $Z_\alpha$ in the direction $e_i$ (which is the same for all $\alpha$) and on the $i$-th coordinate of the center of $Z_\alpha$, \textit{i.e.} on $\alpha_i$.

%Since we are only interested in the modulo $2$ class of $N_i(\alpha)$, we can moreover assume that the length of the zonotopes in the direction $e_i$ is less than or equal to $2\pi$, since a zonotope of length $l + k2\pi$ in direction $e_i$ (for $l\in [0,2\pi ]$ and $k\in \mathbb{N}$) would have $2k$ intersections with  $\{[\theta_i] =0\}$ more than a zonotope of length $l$ with the same center.  INUTILE, MAIS AIDE A VISUALISER. GARDER?

As we only have to consider the coordinate $\alpha_i$ and because of the symmetry of the situation, it is clear that if $\beta \in J_i$, then $c(\beta)\in J_i$. Consider the description (\ref{descriptionImage}) of $Im(1+c_*)$; if $\beta \neq c(\beta)$, we have
$$\partial Z_\beta + \partial Z_{c(\beta)} \in \sum_{\alpha\in\Omega} \Z_2 \cdot [\partial Z_\alpha + \partial Z_{c(\alpha)}]. $$
Thus for any $i\in \{1,\ldots,n \}$ we get

\begin{align*}
\Z_2 \cdot ((1+c_*)(B_i)) + \sum_{\alpha\in \Omega} \Z_2 \cdot [\partial Z_\alpha + \partial Z_{c(\alpha)}] &=\\
\Z_2 \cdot \sum_{\beta\in J_i}[\partial Z_\beta] + \sum_{\alpha\in \Omega} \Z_2 \cdot [\partial Z_\alpha + \partial Z_{c(\alpha)}]= \\
\Z_2 \cdot \sum_{\beta \in J_i \cap F}[\partial Z_{\beta}] + \sum_{\alpha\in \Omega} \Z_2 \cdot [\partial Z_\alpha + \partial Z_{c(\alpha)}],
\end{align*}
where as above $F$ is the set of fixed elements $\beta \in \Omega$.

We need to determine for all $i\in \{1,\ldots,n \}$ the set $J_i \cap F$. Doing so will allow us to prove the following Lemma.

\begin{lemma}\label{LemmaImageHomology}

If $|I^{1,0}| \neq 0$, the image of $1+c_*$ is $\sum_{\alpha\in \Omega} \Z_2 \cdot [\partial Z_\alpha + \partial Z_{c(\alpha)}]$.

If $|I^{1,0}| = 0$,
$$    Im(1+c_*) = \sum_{i\in I^{0,0}}\Z_2 \cdot ((1+c_*)(B_i)) + \sum_{\alpha\in \Omega} \Z_2 \cdot [\partial Z_\alpha + \partial Z_{c(\alpha)}].$$
\end{lemma}

\begin{proof}
Indeed, as noted in Subsection \ref{SubsectionActionIndices}, there are no fixed points if $|I^{1,0}| \neq 0$.

Let us now consider the case $|I^{1,0}| = 0$. 

If $i\in I^{0,0}$,  $\beta\in J_i$ for all $\beta$ such that $\beta_i = 0$, and $\beta\notin J_i$ for all $\beta$ such that $\beta_i =(d_i -1)/2$. In particular, $J_i \cap F $ is the set of cardinality $2^{|I^{0,0}|-1}$ of all $\beta \in \Omega$ such that $\beta_i=0$, $\beta_j \in \{ 0, d_j / 2\}$ if $j\in I^{0,0} - \{i\}$, $\beta_j=0$ if $j\in I^{0,1}$ and $\beta_j=(d_j-1)/2$ if $j\in I^{1,1}$.

If $i\in I^{0,1}$, $\beta\in J_i$ for all $\beta$ such that $\beta_i = 0$. In particular, $J_i \cap F = F$. Remember description (\ref{descriptionHomologie}); we have $\sum_{\alpha \in \Omega}[\partial Z_{\alpha}] =0$ in $H_{n-1}(\CC_X)$.
Then,
$$  \sum_{\alpha \in F}[\partial Z_{\alpha}] = \sum_{\alpha \in F}[\partial Z_{\alpha}] +  \sum_{\alpha \in \Omega}[\partial Z_{\alpha}] \subset \sum_{\alpha\in \Omega} [\partial Z_\alpha + \partial Z_{c(\alpha)}] .$$
This implies that $\Z_2 \cdot ((1+c_*)(B_i)) \subset \sum_{\alpha\in \Omega} \Z_2 \cdot [\partial Z_\alpha + \partial Z_{c(\alpha)}]$.

If $i\in I^{1,1}$, $\beta\notin J_i$ for all $\beta$ such that $\beta_i =(d_i -1)/2$. In particular, $J_i \cap F = \emptyset$ .

Using description (\ref{descriptionImage}) , we can conclude.
\end{proof}

\subsection{Rank of $1+c_*$}

Let us compute the rank of $1+c_*$. We get the following proposition.
\begin{proposition}\label{PropositionRank}
If $|I^{1,0}| \neq 0$,
$$\dim({Im(1+c_*)}) = \frac{d_1\ldots d_n} {2}-1. $$
If $|I^{1,0}| = 0$,
$$\dim({Im(1+c_*)})= |I^{0,0}|+ \frac{d_1\ldots d_n-2^{|I^{0,0}|}}{2}.$$

\end{proposition}

\begin{proof}
We will first compute the dimension of $$\sum_{i=1,\ldots,n}\Z_2 \cdot ((1+c_*)(B_i)) + \sum_{\alpha\in \Omega} \Z_2 \cdot [\partial Z_\alpha + \partial Z_{c(\alpha)}]$$ in 
$$\left(\bigoplus_{i=1,\ldots,n} \Z_2 \cdot B_i\right) \oplus \left(\bigoplus_{\alpha\in \Omega} \Z_2 \cdot [\partial Z_\alpha]\right),$$
then quotient by $Z_2 \cdot \sum_{\alpha\in \Omega} [\partial Z_\alpha]$ to get the dimension in $H_{n-1}(\CC _ X)$ (see Lemma \ref{descriptionHomologie}).

Denote by $\tilde{\Omega}\subset \Omega$ a set defined as such: for each pair of elements $\{\alpha, c(\alpha)\}$, where $\alpha \in \Omega \backslash F$, choose exactly one element to be included in $\tilde{\Omega}$.
Thus $\sum_{\alpha\in\Omega} \Z_2 \cdot [\partial Z_\alpha + \partial Z_{c(\alpha)}] = \sum_{\alpha\in \tilde{\Omega}} \Z_2 \cdot [\partial Z_\alpha + \partial Z_{c(\alpha)}]$.
This sum is clearly a direct sum before quotienting.

If $|I^{1,0}| \neq 0$, then $Im(1+c_*) =\sum_{\alpha\in \Omega} \Z_2 \cdot [\partial Z_\alpha + \partial Z_{c(\alpha)}]= \sum_{\alpha \in \tilde{\Omega}} \Z_2 \cdot [\partial Z_\alpha + \partial Z_{c(\alpha)}]$. When we quotient by $Z_2 \cdot \sum_{\alpha\in \Omega} [\partial Z_\alpha]$, the dimension decreases by exactly one, as $\sum_{\alpha\in \Omega} [\partial Z_\alpha] \in \sum_{\alpha \in \tilde{\Omega}} \Z_2 \cdot [\partial Z_\alpha + \partial Z_{c(\alpha)}]$ (since there is no fixed point). Hence 
$$\dim({Im(1+c_*)})= |\tilde{\Omega}|-1 = \frac{d_1\ldots d_n} {2}-1. $$

If $|I^{1,0}| = 0$, we can see that
$$\sum_{i\in I^{0,0}}\Z_2 \cdot ((1+c_*)(B_i)) + \sum_{\alpha\in \tilde{\Omega}} \Z_2 \cdot [\partial Z_\alpha + \partial Z_{c(\alpha)}]$$
is actually a free sum before quotienting. Indeed, suppose that for each $ \alpha \in \tilde{\Omega} $ (respectively, $i\in I^{0,0}$) there is $\lambda_\alpha \in \Z_2$ (respectively, $\lambda_i \in \Z_2$), not all $0$, such that 
\begin{equation}\label{PreuveIndep}
    \sum_{i\in I^{0,0}}\lambda_i  (1+c_*(B_i)) + \sum_{\alpha\in \tilde{\Omega}} \lambda_\alpha [\partial Z_\alpha + \partial Z_{c(\alpha)}] = 0.
\end{equation}
For $i\in I^{0,0}$, consider $\beta$ defined by $\beta_i=0$, $\beta_j =d_j / 2$ if $j\in I^{0,0} - \{i\}$, $\beta_j=0$ if $j\in I^{0,1}$ and $\beta_j=(d_j-1)/2$ if $j\in I^{1,1}$. As explained in the description of $J_i\cap F$ in the proof of Lemma \ref{LemmaImageHomology}, the class $[\partial Z_\beta]$ only appears in $((1+c_*)(B_i))$ among all the terms of equation (\ref{PreuveIndep}). Thus $\lambda_i$ is necessarily $0$. We can then conclude from the independence of the family $\{[\partial Z_\alpha + \partial Z_{c(\alpha)}]\}_{\alpha \in \tilde{\Omega}}$ before the quotient that all $\lambda_\alpha$ are $0$, which proves our point.

Consider $\beta \in \Omega $ defined by $\beta_j =d_j / 2$ if $j\in I^{0,0}$, $\beta_j=0$ if $j\in I^{0,1}$ and $\beta_j=(d_j-1)/2$ if $j\in I^{1,1}$. Since $[\partial Z_\beta]$ appears in $\sum_{\alpha\in \Omega} [\partial Z_\alpha]$ but not in $\sum_{i\in I^{0,0}}\Z_2 \cdot ((1+c_*)(B_i)) + \sum_{\alpha\in \tilde{\Omega}} \Z_2 \cdot [\partial Z_\alpha + \partial Z_{c(\alpha)}]$ (once again going back to the proof of Lemma \ref{LemmaImageHomology}), we see that
$\sum_{\alpha\in \Omega} [\partial Z_\alpha] \notin \sum_{i\in I^{0,0}}\Z_2 \cdot ((1+c_*)(B_i)) + \sum_{\alpha\in \tilde{\Omega}} \Z_2 \cdot [\partial Z_\alpha + \partial Z_{c(\alpha)}]$. This means that the dimension of $\sum_{i\in I^{0,0}}\Z_2 \cdot ((1+c_*)(B_i)) + \sum_{\alpha\in \tilde{\Omega}} \Z_2 \cdot [\partial Z_\alpha + \partial Z_{c(\alpha)}]$ does not decrease when we quotient.
Thus
$$\dim({Im(1+c_*)})=|I^{0,0}| +|\tilde{\Omega}| = |I^{0,0}|+ \frac{|\Omega|-|F|}{2}= |I^{0,0}|+ \frac{d_1\ldots d_n-2^{|I^{0,0}|}}{2},$$
as stated.
\end{proof}

%% file: GaloisMaximality.tex
\section{Galois maximality}\label{GaloisMaximality}

We are still using the notations of the previous sections. 

%In particular, remember that $X$ is given by the simplicial real polynomial $P(z) = \sum _{i=1,\ldots,n}  \varepsilon_i c_{i} z^{A_i} +  c_0$.

%To determine the conditions under which $X$ is Galois maximal, we need to compare $\dim_{\Z _2}\left(\frac{Ker(1+c_*)}{Im(1+c_*)}\right)$ and $ \dim_{\Z _2}(H_*(\R X))$. We consider $H_*(\R X)$ first, with the following lemma (well-known to anyone familiar with the combinatorial case of Viro's Patchworking method). 

To prove Theorem \ref{MainResult}, we need to compare $\dim_{\Z _2}\left(\frac{Ker(1+c_*)}{Im(1+c_*)}\right)$ (where $1+c_*: H_{*}(\CC_ X)\longrightarrow H_{*}(\CC _X)$) and $ \dim_{\Z _2}(H_*(\R X))$. We consider $H_*(\R X)$ first, with the following lemma. 

\begin{lemma}\label{LemmaRealPartHomology}

The real part $\R X \subset (\R^ *)^n$ consists of $2^n$ contractible connected components if $\delta = \delta(\varepsilon) \in {\Z_2}^n$ is not in the image of $A\cdot:{\Z_2}^n \longrightarrow {\Z_2}^n$, and of $2^n - 2 ^{n-\rank (A)}$ contractible connected components if it is. Thus $H_*(\R X)$ is of $\Z_2$-dimension $2^n$ if $\delta \notin Im(A)$ and of $\Z_2$-dimension $2^n - 2 ^{n-\rank (A)}$ if $\delta \in Im(A)$.
\end{lemma}

\begin{proof}
Consider the $2^n$ quadrants of $(\R^ *)^n$, and let $\gamma=(\gamma_1,\ldots, \gamma_n ) \in \{1,-1 \}^n$ index the quadrant $Q_\gamma := \{(x_1,\ldots,x_n) \in  (\R^ *)^n | x_1 \gamma_1 >0, \ldots, x_n \gamma_n >0 \}$. Consider as in Section \ref{Coamoebas} the real polynomial  $P_L(z) = \sum _{i=1,\ldots,n}  \varepsilon_i c_{i} z_i + c_0$ associated to $P(z) = \sum _{i=1,\ldots,n}  \varepsilon_i c_{i} z^{A_i} +  c_0$. Then the map
\begin{align*}
   \Phi _{A^{-1}} : \quad  Q_{(1,\ldots,1)}  & \to  Q_{(1,\ldots,1)} \\
    z & \mapsto  z^{A^{-1}},
\end{align*}
where we extend the notation to include rational exponents, is a well defined homeomorphism that maps $\R X_{P_L} \cap  Q_{(1,\ldots,1)}$ to  $\R X \cap  Q_{(1,\ldots,1)}$. In particular, $\R X \cap  Q_{(1,\ldots,1)}$ is empty if $\varepsilon_i=1$ for all $i\in \{1,\ldots,n \}$, and is non-empty and contractible otherwise.
Now observe that $\R X \cap  Q_{\gamma}$ is isomorphic to $\R X_{P_ \gamma} \cap  Q_{(1,\ldots,1)}$, where $P_\gamma (z) := \sum _{i=1,\ldots,n}  \varepsilon_i \gamma^{A_i} c_{i} z^{A_i} +  c_0$. Since $ \delta (\varepsilon_1 \gamma^{A_1}, \ldots, \varepsilon_n \gamma^{A_n})= \delta(\varepsilon) + A\cdot \delta(\gamma)$, we see that $\R X \cap  Q_{\gamma}$ is empty if $ \delta(\varepsilon) = A\cdot \delta(\gamma)$, and non-empty and contractible otherwise.

\end{proof}

The next lemma links the conditions of Theorem \ref{MainResult}, Lemma \ref{LemmaImageHomology} and Proposition \ref{PropositionRank}. 
\begin{lemma}\label{LemmaEquivalenceConditions}
 The real part $\R X$ intersects non-trivially each quadrant of $(\R^*)^n$ if and only if  $\delta \notin Im(A)$ if and only if $|I^{1,0}| \neq 0$.
\end{lemma}
\begin{proof}
The first equivalence comes from Lemma \ref{LemmaRealPartHomology}. Then
\begin{align*}
\delta \in Im(A) \iff \exists x \in {\Z_2}^n \: s.t. \: A\cdot x=G^{-1}\cdot D \cdot x=\delta \iff  \\
\exists x \in {\Z_2}^n \: s.t. \: D\cdot x=G\cdot \delta= \delta ^G
\iff \delta ^G \in Im(D) \iff \\
(\delta ^G)_i = 0 \: \forall i \: s.t. \: d_i= [0]_2
\iff |I^{1,0}| = 0.
\end{align*}
\end{proof}

We are now ready to prove the main result.
\begin{proof}[Proof of Theorem \ref{MainResult}]
%It is easy to show, using an open version of Lefschetz's hyperplan theorem with $\C X \subset (\C^*)^n$ proved in \cite{OpenLefschetz}, that $c_*$ acts trivially on $H_k(\C X)$ for $k<n-1$, because it acts trivially on $(\C^*)^n$. As $H_k(\C X)=0$ for $k>n-1$, the only degree of interest is the degree $n-1$.
%We have computed in the previous section the rank of $1+c_*: H_{n-1}(\CC_ X)\longrightarrow H_{n-1}(\CC _X)$ in degree $n-1$. As mentioned in Section \ref{Coamoebas}, there is a homotopy equivalence compatible with the conjugation between $ \CC_ X$ and $ \C X$, which is why we don't hesitate to use the same notation for $1+c_*: H_{n-1}(\CC_ X)\longrightarrow H_{n-1}(\CC _X)$ and $1+c_*: H_{n-1}(\C X)\longrightarrow H_{n-1}(\C X)$ (as they are identified by the isomorphism $H_{n-1}(\C X)\cong H_{n-1}(\CC_ X)$). In particular,
%\begin{align*}
%\dim (Im(1+c_* :H_{n-1}(\C X)\rightarrow H_{n-1}(\C X) )) =\\ \dim (Im(1+c_* :H_{n-1}(\CC_ X)\rightarrow H_{n-1}(\CC_ X) )).
%\end{align*}

As mentioned at the beginning of Subsection \ref{SubsectionHomologyCoamoeba}, $1+c_*: H_{k}(\CC_ X)\longrightarrow H_{k}(\CC _X)$ is trivial for $k\neq n-1$.
Moreover, we know that $H_k(\CC_X)$ is isomorphic to $H_k(T^n)$ for $k\neq n-1$, and that $dim(H_{n-1}(\CC_X))=n + d_1\ldots d_n -1$ (as is shown by Lemma \ref{descriptionHomologie}).

Therefore, from Proposition \ref{PropositionRank} and Lemmas \ref{LemmaRealPartHomology} and \ref{LemmaEquivalenceConditions}, we see that if $\R X$ intersects non-trivially each quadrant of $(\R^*)^n$,
\begin{align*}
\dim \left(\frac{Ker(1+c_*)}{Im(1+c_*)}\right) = \dim(H_*(\C X))-2\dim (Im(1+c_*)) =\\
2^n-1+(d_1\ldots d_n -1) -2( \frac{d_1\ldots d_n} {2}-1)=2^n= \dim(H_*(\R X)). 
\end{align*}
Hence $X$ is Galois maximal.

Otherwise, we have by definition of the sets $I^{\pm 1, \pm 1}$ in Section \ref{Coamoebas} that $\rank _ {\Z _2}(\Delta _ P) = \rank _ {\Z _2}(A)= |I^{0,1}|+|I^{1,1}| = n- |I^{0,0}|$ (since $|I^{1,0}| = 0)$ and
\begin{align*}
\dim \left(\frac{Ker(1+c_*)}{Im(1+c_*)}\right) = \dim(H_*(\C X))-2\dim (Im(1+c_*)) =\\
2^n-1+(d_1\ldots d_n -1) -2( |I^{0,0}| +\frac{d_1\ldots d_n - 2^{|I^{0,0}|}} {2})=\\
2^n-2-2(n-\rank _{\Z_2}(A)) + 2^{n-\rank _{\Z_2}(A)},
\end{align*}
hence 
\begin{flalign*}
& \dim \left(\frac{Ker(1+c_*)}{Im(1+c_*)}\right) -\dim (H_*(\R X)) = & \\
 & 2^n-2-2(n-\rank _{\Z_2}(A)) + 2^{n-\rank _{\Z_2}(A)} -(2^n - 2 ^{n-\rank (A)}) = &\\
& 2(2^{n-\rank _{\Z_2}(A)}-1- (n-\rank _{\Z_2}(A))) &.
\end{flalign*}

This ends the proof.
\end{proof}